\documentclass[a4paper,11pt,final]{amsart} 
\usepackage{amssymb,amsmath}
\usepackage{color}
\usepackage{a4wide}

\numberwithin{equation}{section}
\newtheorem{theorem}{Theorem}[section]
\newtheorem{proposition}[theorem]{Proposition}

\newtheorem{lemma}[theorem]{Lemma}

\newtheorem{remark}[theorem]{Remark}

\newcommand{\R}{{\mathbb R}}
\newcommand{\Rn}{{\mathbb R}^N}
\newcommand{\Sn}{{\mathbb S}^N}

\def\grad{\nabla}

\def\bye{\end{document}}
\def\by{\end{proof}\bye}

\def\beq{\begin{equation}}
\def\eeq{\end{equation}}

\newcommand{\Pmk}{\mathcal{P}^-_{k}}
\newcommand{\Ppk}{\mathcal{P}^+_{k}}
\newcommand{\Ppmk}{\mathcal{P}^{\pm}_{k}}

\title[\textbf{Liouville theorems}]{\textbf{Liouville theorems for a family of very degenerate elliptic non linear operators.}}
\author{Isabeau Birindelli, Giulio Galise, Fabiana Leoni}%\\ Sapienza Universit\`a di Roma}
\date{}
\address{Dipartimento di Matematica\newline
\indent Sapienza Universit\`a  di Roma \newline
 \indent   P.le Aldo  Moro 2, I--00185 Roma, Italy.}
 \email{isabeau@mat.uniroma1.it}
 \email{galise@mat.uniroma1.it}
\email{leoni@mat.uniroma1.it}

\begin{document}

\keywords{Fully nonlinear degenerate elliptic equations, viscosity solutions, Liouville type results}
\subjclass[2010]{35J60, 35J70, 35B53}
\begin{abstract} We prove nonexistence results of Liouville type for nonnegative viscosity solutions of some equations involving the fully nonlinear degenerate elliptic  operators $\Ppmk$, defined respectively as the sum of the largest and the smallest $k$ eigenvalues of the Hessian matrix. For the operator $\Ppk$ we obtain results analogous to those which hold for the Laplace operator in space dimension $k$. Whereas, owing to the stronger degeneracy of the operator $\Pmk$, we get totally different results.
 \end{abstract}

\maketitle
\section{Introduction} We  study the existence of nonnegative viscosity either solutions or supersolutions of fully nonlinear elliptic equations of the form
\begin{equation}\label{eq0}
F(D^2u(x))=0\quad\text{in}\quad\mathbb R^N\, ,
\end{equation}
or
\begin{equation}\label{eq1}
F(D^2u(x))+u^p(x)=0\quad\text{in}\quad\mathbb R^N,
\end{equation}
or 
\begin{equation}\label{eq2}
F(D^2u(x))+u^p(x)=0\quad\text{in}\quad\mathbb R^N_+,
\end{equation}
where $p>0$, $\mathbb R^N_+$ is any halfspace and $F$ is either $\Pmk$ or $\Ppk$. 
Here $\Pmk$ and $\Ppk$ are second order degenerate elliptic operators defined, for a positive integer $k\in[1,N]$ and any symmetric $N\times N$ matrix $X$, by the  partial sums
\begin{equation}\label{Pk}
\Pmk(X)=\sum_{i=1}^k\lambda_i(X)\quad \mbox{and}\quad  \Ppk(X)=\sum_{i=N-k+1}^N\lambda_i(X)
\end{equation}
 of the ordered eigenvalues $\lambda_1(X)\leq\cdots\leq \lambda_N(X)$  of $X$.  

When $k=N$ these operators coincide with the Laplacian. In this  case,  for equation \eqref{eq0} the results go back to the classical ones of  Cauchy 
and Liouville, whereas, for equations \eqref{eq1} and \eqref{eq2} where the reaction term $u^p$ is included, they have been started respectively by Gidas and Spruck in their acclaimed work  \cite{gs1} and by Berestycki, Capuzzo Dolcetta and Nirenberg in \cite{BCDN} and by Bandle and Levine in \cite{BL}.

In the present paper we focus on the degenerate cases $k<N$.  We will see that 
the existence or lack of existence is quite different depending on 
whether one considers $\Pmk$ or $\Ppk$. 
We collect our main results  in the following statements and then we will give some comments. 

The first theorem concerns equation \eqref{eq0}. 
\begin{theorem}\label{liou}Let $k<N$ be a positive integer.
\begin{enumerate}
\item For $F\equiv\Pmk$, there exist nonnegative classical solutions of \eqref{eq0} which are not constants.
\item\label{5} For $F\equiv\Ppk$ if $u$ is a nonnegative viscosity supersolution of \eqref{eq0} and $k=1$ or $k=2$, then $u$ 
is constant. If $k>2$ there are nonnegative classical supersolutions of \eqref{eq0} which are not constants.
\end{enumerate}
\end{theorem}

For equation \eqref{eq1}, we separate the results concerning $F\equiv\Pmk$ and $F\equiv \Ppk$.

\begin{theorem}\label{P1}
Let $k<N$ be a positive integer and let $F\equiv\Pmk$.
\begin{enumerate}
\item\label{2} For any $p>0$ there exist nonnegative viscosity solutions $u\not\equiv0$ of \eqref{eq1}.
\item\label{3} For any $p\geq1$ there exist positive classical solutions  of \eqref{eq1}.
\item\label{4} For $p<1$ there are no positive viscosity supersolutions of \eqref{eq1}.

%\begin{equation}\label{eq2.1}
%\Pmk(D^2u(x))+u^p(x)=0\quad\text{in}\;\mathbb R^N\ .
%\end{equation}

%\item\label{3} For any $p\geq 1$ there exists a  viscosity super solution $u$ of
%\begin{equation}\label{eq2}
%\Pmk(D^2u(x))+u^p(x)\leq0\quad\text{in}\;\mathbb \{x_1>0\},\ u(0,x^\prime)=0\ .
%\end{equation}
%such that $u(x_1,x^\prime)>0$ in $\{x_1>0\}$.
\end{enumerate}
\end{theorem}

 We  observe that any nonnegative supersolution of \eqref{eq1} is in turn a supersolution of \eqref{eq0}. Hence, the Liouville property for equation \eqref{eq1} with $F\equiv\Ppk$   directly follows from Theorem \ref{liou}-\eqref{5} in the cases $k=1$ and $k=2$. For the remaining cases we have the following result.

\begin{theorem}\label{P2}Let $2<k<N$ be a positive integer and let $F\equiv\Ppk$ .
\begin{enumerate}
\item\label{6} For $0<p\leq \frac{k}{k-2}$ the only nonnegative viscosity supersolution of \eqref{eq1} is $u\equiv0$.
\item\label{7} For  $p>\frac{k}{k-2}$ there exist  positive classical supersolutions  of \eqref{eq1}.
\item\label{8} For $p\geq\frac{k+2}{k-2}$ there exist positive classical solutions of \eqref{eq1}.
\item\label{9} For $p\in(\frac{k}{k-2},\frac{k+2}{k-2})$ there are no radial positive classical solutions of \eqref{eq1}.
\end{enumerate}
\end{theorem}

As far as solutions in the halfspace are concerned, our results read as follows.\\ Without loss of generality, henceforth we set
$$
\mathbb R_+^N=\left\{(x',x_N)\in\Rn:x_N>0\right\}\, .
$$

\begin{theorem} \label{10} Let $k<N$ be a positive integer. For any $p>0$ there exist nonnegative  bounded viscosity solutions of
\begin{equation}\label{halfspace}
\Pmk(D^2u(x))+u^p(x)=0\quad\mbox{in } \mathbb R_+^N\, ,\quad u=0\quad\mbox{on }\partial\mathbb R^N_+
\end{equation}
and such that 
\begin{equation}\label{limsup}
\limsup_{x_N\to+\infty}u(x',x_N)>0\quad\text{uniformly w.r.t. $x'\in\mathbb R^{N-1}$\, .}
\end{equation}
\end{theorem}

\begin{theorem}\label{Pk+ halfspaces}Let $k<N$ be a positive integer. If $k=1$ and $p>0$ or $k>1$ and $p<\frac{k+1}{k-1}$, then there does not exist positive viscosity supersolutions of the equation
\begin{equation*}
\Ppk(D^2u(x))+u^p(x)=0\quad\mbox{in}\quad\mathbb R_+^N.%=\left\{(x',x_N)\in\Rn:x_N>0\right\}.
\end{equation*}
\end{theorem}

Let us recall that the  operators $\mathcal{P}^\pm_{k}$ have been initially introduced in connection with Riemannian geometry. In 
particular, they have been considered  by Sha in 
 \cite{Sha} when studying  
 $k$ convex manifolds, while the case of manifolds with partially positive curvature was seen by Wu in \cite{Wu}.  They have been exhibited also in \cite[Example 1.8]{CIL}, as examples of fully nonlinear degenerate elliptic operators,
and they also appear in the level set approach to mean curvature flow of manifolds with arbitrary codimension developed  by Ambrosio and Soner  in \cite{AS}. More recently, in a PDE context,
we wish to recall the works of Harvey and Lawson  \cite{HL1, HL2} that have given a new geometric interpretation of 
solutions,  while Caffarelli, Li and Nirenberg in \cite{CLN1} in their study of degenerate elliptic equations, give 
some  results concerning removable singularities along smooth manifolds for Dirichlet problems associated to $\Pmk$. See also \cite{AGV,GV} for the extended version of the maximum principle and \cite{BGI} for regularity and existence of the principal eigenfunctions. We further recall that existence of entire sub/supersolutions of equations 
 involving   $\Ppmk$ and having different lower order terms have  been studied in \cite{CDLV1, CDLV2}.
 
A connection between the existence of solutions relative to $\mathcal{P}^\pm_{k}$ and the existence of solutions relative to Laplace operator in dimension $k$  may be expected by the definition of the operators $\mathcal{P}^\pm_{k}$ itself. As a general fact, we notice that if $v$ is a function of $k$ variables and we consider it as a function of $N>k$ variables just by setting $u(x_1,\ldots ,x_N)=v(x_1,\ldots ,x_k)$, then one has
 $$
\Ppk (D^2u)\geq \Delta v\geq \Pmk (D^2u)\, ,
$$
and $\Ppk (D^2u)= \Delta v$ if and only if $v$ is convex, as well as $\Pmk (D^2u)= \Delta v$ if and only if $v$ is concave. This remark does not lead to any immediate extension of existence/non existence results for Laplace operator in dimension $k$  to analogous results for operators $\mathcal{P}^\pm_{k}$. On the other hand, if $u$ is a function of $N$ variables satisfying $\Ppk(D^2u)\leq0$ (or $\Pmk(D^2u)\geq0$), then $\Delta u\leq0$ (respectively $\Delta u\geq0$). This implies that the nonexistence results relative to supersolutions of the Laplace operator immediately extend to nonexistence results for $\Ppk$. In the present paper we prove stronger results. The thresholds we found for $\Ppk$ are those that correspond to the Laplace operator in dimension $k$.  Indeed, we can adapt the techniques developed in \cite{CL,L} for fully nonlinear uniformly elliptic operators in $\R^k$ to get our results of Theorem \ref{liou}-(2), Theorem \ref{P2} and Theorem \ref{Pk+ halfspaces}, each time paying attention to the order of the eigenvalues of the Hessian matrix of the involved test functions.

\noindent  On the other hand, the statements of Theorem \ref{liou}-(1), Theorem \ref{P1} and Theorem \ref{10} for operator $\Pmk$ drastically differ from the analogous ones in the uniformly elliptic case, and they are obtained by using ad hoc constructions of  explicit solutions. The diversity of results for the two operators is not surprising if one takes into account the stronger degeneracy of  the  operator $\Pmk$ with respect to supersolutions, which causes for example the failure of the strong minimum principle, see e.g.  \cite{BGI}.

Let us conclude with a final remark. In Theorem \ref{P2} we excluded the existence of radial solutions of equation \eqref{eq1} for $p\in(\frac{k}{k-2},\frac{k+2}{k-2})$, but the question of existence of non radial solutions is left open. This is clearly related to  the more general question of radial symmetry of positive solutions for the equation
$$\Ppk(D^2u)+f(u)=0\quad (\mbox{or}\  \Pmk(D^2u)+f(u)=0)\quad \mbox{in }\ \R^N\, .$$
Let us mention that the usual proof for semilinear elliptic equations based on  the moving planes method,   highly relying on the
strong maximum principle, seems not to apply to the operator $\Ppk$.

\section{Preliminaries}
The operators $\Ppmk$ are elliptic second order operators which degenerate in any direction, i.e. for any $\xi\in\Rn$ such that $|\xi|=1$ then
\begin{equation}\label{degenerate}
\min_{X\in\Sn}\left(\Ppmk(X+\xi\otimes\xi)-\Ppmk(X)\right)=0.
\end{equation}
To prove \eqref{degenerate}  just take $X=0$ in the case of $\Pmk$ and  $X=-\xi\otimes\xi$ for $\Ppk$,  then use the fact that ${\rm spec}(\xi\otimes\xi)=\left\{0,\ldots,0,1\right\}$. They can be equivalently defined either by the partial sums \eqref{Pk} or by the representation formulas
\begin{equation*}
\begin{split}
\Pmk(X)&=\min\left\{\sum_{i=1}^k\left\langle X\xi_i,\xi_i\right\rangle\,|\,\text{$\xi_i\in\Rn$ and $\left\langle\xi_i,\xi_j\right\rangle=\delta_{ij}$, for $i,j=1,\ldots,k$}\right\}\\
\Ppk(X)&=\max\left\{\sum_{i=1}^k\left\langle X\xi_i,\xi_i\right\rangle\,|\,\text{$\xi_i\in\Rn$ and $\left\langle\xi_i,\xi_j\right\rangle=\delta_{ij}$, for $i,j=1,\ldots,k$}\right\},
\end{split}
\end{equation*}
see \cite[Lemma 8.1]{CLN1}. It is then easy to see the superadditivity/subadditivity properties 
\begin{equation}\label{suba}
\Pmk(Y)\leq{\mathcal P}^\pm_k(X+Y)-{\mathcal P}^\pm_k(X)\leq\Ppk(Y).
\end{equation}

\begin{theorem}[\textbf{Strong Minimum Principle}]\label{SMP} Let $\Omega\subset\Rn$ be a domain and let $u\in LSC(\Omega)$ be a viscosity supersolution of $\Ppk(D^2u)=0$ in $\Omega$. If $u$ achieves its minimum in the interior of $\Omega$, then $u$ is a constant.
\end{theorem}
\noindent Note that the previous theorem fails for $\Pmk$, see \cite{BGI}.

We conclude the section with two lemmas. They will be used respectively in sections \ref{Whole space}-\ref{Half space}.

\begin{lemma}\label{rem1} 
 Let $f\in C^1([0,+\infty))$ such that $f(u)f^\prime(u)\geq 0$  for $u\geq0$. Then,  for any $R\leq+\infty$ and any solution $u\in C^2([0,R))$  of
\begin{equation}\label{eqremark}
k\frac{u^\prime(r)}{r}=-f(u)\quad \mbox{in}\quad (0,R)\quad\mbox{and}\quad u'(0)=0,
\end{equation}
the radial function $v(x)=u(|x|)$ is a solution in $B_R$ of
$$\Pmk(D^2v)+f(v)=0.$$
\end{lemma}
\begin{proof}
Since $$u^{\prime\prime}(r)=\frac{-f(u)}{k}+\frac{r^2}{k^2}f^\prime(u)f(u)\geq \frac{u^\prime(r)}{r}$$  using \eqref{eqremark}, we deduce that
$$\Pmk(D^2v(x))=k\frac{u^\prime(|x|)}{|x|}=-f(v(x))\quad\text{ if $x\neq0$}$$ and $$\Pmk(D^2v(0))=ku''(0)=-f(v(0)).$$
\end{proof}

In the section \ref{Half space} we will use explicit classical subsolutions of type $\varphi(x)=x_Nf(|x|)$, whose Hessian is
%\begin{split}
$$D^2\varphi(x)=f'(|x|)\left[e_N\otimes\frac {x}{|x|}+\frac {x}{|x|}\otimes e_N\right]+x_N\left[\left(f''(|x|)-\frac{f'(|x|)}{|x|}\right)\frac {x}{|x|}\otimes\frac {x}{|x|}+\frac{f'(|x|)}{|x|}I\right].
%\end{split}
$$
The computation of the eigenvalues is  straightforward. We have (see also \cite{L,L2}) 

\begin{lemma}\label{eigenvalues}
Let $\varphi(x)=x_N f(|x|)$ such that $f\in C^2((0,+\infty))$ and $x\in\mathbb R^N_+$. Set 
$$
b=x_N\left(f''(|x|)-\frac{f'(|x|)}{|x|}\right),\quad c=f'(|x|),\quad d=x_N\frac{f'(|x|)}{|x|}\,.
$$
 Then the eigenvalues of the $D^2\varphi (x)$ are:
\begin{itemize}
	\item $d$ with multiplicity (at least $N-2$) and eigenspace  ${\rm span}\left\{ e_N,x\right\}^{^\bot}$;
	\item $\displaystyle d+\frac{b+2c\frac{x_N}{|x|} \pm\sqrt{\left(b+2c\frac{x_N}{|x|}\right)^2+4c^2\left(1-\left(\frac{x_N}{|x|}\right)^2\right)}}{2}$\,. %and eigenspace spanned by \\ $\left\{\frac{x}{|x|}+\frac{-b+\sqrt{b^2+4c\left(b\left\langle e_N,x \right\rangle+c\right)}}{2(b\left\langle e_N,x \right\rangle+c)}e_N\right\}$ if $b\left\langle e_N,\frac{x}{|x|} \right\rangle+c\neq0$, \\$\left\{\frac{x}{|x|}+\frac{c}{b}e_N\right\}$ otherwise
	%\item $\displaystyle d+\frac{b+2c\left\langle e_N,\frac{x}{|x|} \right\rangle-\sqrt{b^2+4c\left(b\left\langle e_N,\frac{x}{|x|} \right\rangle+c\right)}}{2}$ and eigenspace spanned by \\ $\left\{\frac{x}{|x|}+\frac{-b-\sqrt{b^2+4c\left(b\left\langle e_N,x \right\rangle+c\right)}}{2(b\left\langle e_N,x \right\rangle+c)}e_N\right\}$ if $b\left\langle e_N,\frac{x}{|x|} \right\rangle+c\neq0$, \\$\left\{\frac{x}{|x|}+\frac{c}{b}e_N,e_N\right\}$ otherwise.
\end{itemize}
In particular the ordered eigenvalue of $D^2\varphi$ are
\begin{equation*}
\begin{split}
\lambda_1(D^2\varphi)&=d+\frac{b+2c\frac{x_N}{|x|}-\sqrt{\left(b+2c\frac{x_N}{|x|}\right)^2+4c^2\left(1-\left(\frac{x_N}{|x|}\right)^2\right)}}{2}\\
\lambda_2(D^2\varphi)&=\ldots=\lambda_{N-1}(D^2\varphi)=d\\
\lambda_{N}(D^2\varphi)&= d+\frac{b+2c\frac{x_N}{|x|}+\sqrt{\left(b+2c\frac{x_N}{|x|}\right)^2+4c^2\left(1-\left(\frac{x_N}{|x|}\right)^2\right)}}{2}\,.
\end{split}
\end{equation*}

% if $c^2\left(1-\left(\frac{x_N}{|x|}\right)^2\right)=0$ and $b+2c\frac{x_N}{|x|}\neq0$, then the eigenvalues are $d$ with multiplicity $N-1$ and $d+b+2c\frac{x_N}{|x|}$, which is simple.
\end{lemma}

\section{Entire solutions}\label{Whole space}
\subsection{Entire solutions of $\Ppmk$} 
In this subsection we prove Theorem \ref{liou}.
Concerning $F\equiv\Pmk$, the result is immediate since for any nonconstant convex function $f\in C^2(\R;[0,+\infty))$  and $i\in\left\{1,\ldots,N\right\}$, $u(x)=f(x_i)$ is an entire nontrivial solution of
$$ \Pmk (D^2u)=0.$$
\noindent Instead, when $F\equiv\Ppk$ the proof is more involved and uses the three circles of Hadamard principle. The fundamental solutions of the operator $\Ppk$, namely the classical radial solutions of the equation
\begin{equation}\label{eq4}
\Ppk(D^2\varphi)=0\quad\mbox{in}\quad\Rn\backslash\left\{0\right\},
\end{equation}
are defined by
\begin{equation}\label{fond_sol}
\phi(x)=
\begin{cases}
-c_1|x|+c_2 & \text{if $k=1$}\\
-c_1\log|x|+c_2 & \text{if $k=2$}\\
c_1|x|^{2-k}+c_2 & \text{if $k>2$},
\end{cases}
\end{equation}
with constants $c_1\geq0$ and $c_2\in\R$. We notice that differently from the uniformly elliptic case $k=N$, there are not concave and strictly increasing radial solutions of \eqref{eq4}.

Henceforth for any supersolution $u\in LSC(\overline B_R)$  of $$\Ppk(D^2u)=0\quad\mbox{ in}\quad B_R$$
we set $m(r)=\min_{x\in\overline B_r}u(x)$. By definition the function $r\in[0,R]\mapsto m(r)$ is nonincreasing and lower semincontinuous.
Following \cite[Theorem 3.1]{CL} we have a nonlinear version of  Hadamard three circles theorem for $\Ppk$.

\begin{theorem}\label{three circles}
Let $u\in LSC(\Omega)$ be a viscosity supersolution of $\Ppk(D^2u)=0$ in a domain $\Omega\supseteq\overline B_{r_2}$. Then for every fixed $0<r_1<r_2$ and any $r_1\leq r\leq r_2$ one has
\begin{equation}\label{Hadamard}
m(r)\geq
\begin{cases}
\displaystyle\frac{m(r_1)(r_2-r)+m(r_2)(r-r_1)}{r_2-r_1} & \text{if $k=1$}\\
\displaystyle\frac{m(r_1)\log\left(\frac{r_2}{r}\right)+m(r_2)\log\left(\frac{r}{r_1}\right)}{\log\left(\frac{r_2}{r_1}\right)}  & \text{if $k=2$}\\
\displaystyle\frac{m(r_1)(r^{2-k}-r_2^{2-k})+m(r_2)(r_1^{2-k}-r^{2-k})}{r_1^{2-k}-r_2^{2-k}}  & \text{if $k>2$}.
\end{cases}
\end{equation}
\end{theorem}   

In the cases $k=1$ and $k=2$ the fundamental solutions blow down to $-\infty$ for $|x|\to+\infty$. This  allows us to obtain the following Liouville type result, more general of that expressed in Theorem \ref{liou}-\eqref{5} since we are not going to assume $u\geq0$. 
\begin{theorem}\label{Liouville}
Let $u\in LSC(\Rn)$ be a viscosity supersolution of
\begin{equation}\label{eq6}
\Ppk(D^2u)=0\quad\mbox{in}\quad\Rn
\end{equation}
for $k=1$ or $k=2$. Assume that
\begin{equation}\label{Liouville_hypothesis1}
\begin{cases}
\displaystyle\limsup_{r\to+\infty}\frac{m(r)}{r}\geq0\quad\text{if $k=1$}\\
\displaystyle\limsup_{r\to+\infty}\frac{m(r)}{\log r}\geq0\quad\text{if $k=2$}.
\end{cases}
\end{equation}
Then	$u$ is constant. 
\end{theorem}
\begin{proof}[\rm\textbf{Proof.}]
Send $r_2\to+\infty$ in \eqref{Hadamard} and use the assumption \eqref{Liouville_hypothesis1} to obtain $m(r_1)=m(r)$ for any $0<r_1<r$. By lower semicontinuity $u(0)=\lim_{r_1\to0}m(r_1)=m(r)$, hence the strong minimum principle yields $u\equiv u(0)$.
\end{proof}

To finish the proof of Theorem \ref{liou}-\eqref{5} we exhibit  nontrivial bounded supersolutions of \eqref{eq6} in the case $k>2$:
$$
u(x)=\begin{cases}
\frac18\left(k(k-2)|x|^4-2(k^2-4)|x|^2+k(k+2)\right)& \text{if $|x|\leq1$}\\
|x|^{2-k} & \text{if $|x|>1$}.
\end{cases}
$$

\begin{remark}
\rm The assumption \eqref{Liouville_hypothesis1} cannot be weakened. As a matter of fact for any $\varepsilon>0$ the functions
$$
u_1(x)=\begin{cases}
\frac\varepsilon8\left(|x|^4-6|x|^2-3\right)& \text{if $|x|\leq1$}\\
-\varepsilon |x| & \text{if $|x|>1$}
\end{cases}\quad
\mbox{and}\quad
u_2(x)=\begin{cases}
\varepsilon\left(\frac{|x|^4}{4}-|x|^2+\frac34\right)& \text{if $|x|\leq1$}\\
-\varepsilon \log|x| & \text{if $|x|>1$}
\end{cases}
$$
are respectively nontrivial classical solutions of $\Ppk(D^2u)=0$ in $\Rn$ for $k=1$ and $k=2$. Nevertheless
$$
\lim_{r\to+\infty}\frac{m(r)}{r}=-\varepsilon\quad\mbox{and}\quad\lim_{r\to+\infty}\frac{m(r)}{\log r}=-\varepsilon.
$$
\end{remark}

We finally observe that Theorem \ref{three circles} gives (for $r_2\to+\infty$) the following

\begin{proposition}\label{eq7}
Let $u$ be a nonnegative supersolution of \eqref{eq6}. Then for $k\geq2$ the map
\begin{equation}\label{monotonicity}
r\in[0,+\infty)\mapsto m(r)r^{k-2}
\end{equation}
 is nondecreasing.
\end{proposition}

\subsection{Entire solutions of  $\Pmk(D^2u)+u^p=0$.}

\begin{proof}[\rm\textbf{Proof of Theorem \ref{P1}}]

\noindent\textbf{(\ref{2})} Let $p\in(0,1)$ and let $R>0$. We prove that the radial function
$$
u(x)=
\begin{cases}
{\left[\frac{1-p}{2k}\left(R^2-|x|^2\right)\right]}^\frac{1}{1-p} & \text{if $|x|\leq R$}\\
0 & \text{elsewhere}.
\end{cases}
$$
is a nonnegative viscosity solution of \eqref{eq1}.\\ First we note that $u\in C^1(\Rn)\cap C^\infty\left(\mathbb R^N\backslash\left\{x:\,|x|=R\right\}\right)$  and $u\in C^2(\Rn)$ if $p>\frac12$. Moreover by a straightforward computation  $u$ is a classical solution of $$\Pmk(D^2u)+u^p=0\quad\text{in}\;\mathbb R^N\backslash\left\{x:\,|x|=R\right\}.$$ Now we prove that $u$  satisfies this equation also in $\left\{x:\,|x|=R\right\}$ in the viscosity sense. Fix $x_0$ such that $|x_0|=R$.

\smallskip
\emph{u is a subsolution.} If $p<\frac12$ there are no $C^2$ test functions touching $u$ by above at $x_0$, so we have nothing to prove. If otherwise $p\geq\frac12$ and  $\varphi\in C^2(\mathbb R^N)$ is such that 
$$
0=(u-\varphi)(x_0)\geq(u-\varphi)(x)\qquad\forall x\in B_{\delta}(x_0)
$$
for some positive $\delta$, then $\varphi$ has a local minimum point at $x_0$ and hence $D^2\varphi(x_0)\geq0$. It follows $$\Pmk(D^2\varphi(x_0))+\varphi^p(x_0)\geq0.$$

\smallskip
\emph{u is a supersolution.} Suppose instead that
$$
0=(u-\varphi)(x_0)\leq(u-\varphi)(x)\qquad\forall x\in B_{\delta}(x_0)
$$
for some positive $\delta$. We claim that 
\begin{equation}\label{3eq3}
\lambda_{N-1}(D^2\varphi(x_0))\leq0,
\end{equation}
from which the conclusion follows. To prove \eqref{3eq3} we use the variational characterization (Courant-Fischer minmax theorem)
$$
\lambda_{N-1}(D^2\varphi(x_0))=\min_{V}\max_{\xi\in V:\,|\xi|=1}\left\langle D^2\varphi(x_0)\xi,\xi\right\rangle
$$ 
where the minimum is taken over all possible $(N-1)$-dimensional subspaces $V$ of $\mathbb R^N$. 
Let $W:=\left\langle x_0\right\rangle^\bot$. Clearly $\dim W=N-1$. Moreover for any $\xi\in W$ such that $|\xi|=1$ and $t\in(-\delta,\delta)$, we have
\begin{equation*}
\begin{split}
0=u(x_0+t\xi)&\geq\varphi(x_0+t\xi) \\
&=\varphi(x_0)+\left\langle D\varphi(x_0),t\xi\right\rangle+\frac12\left\langle D^2\varphi(x_0)t\xi,t\xi\right\rangle+o(t^2)\\
&=\frac12t^2\left\langle D^2\varphi(x_0)\xi,\xi\right\rangle+o(t^2).
\end{split}
\end{equation*}
Here we have used the facts that $\varphi(x_0)=u(x_0)=0$ and $D\varphi(x_0)=Du(x_0)=0$ since $\varphi$ touches $u$ by below at $x_0$. 
Dividing by $t^2\neq0$ and letting $t\to0$, we deduce that $\left\langle D^2\varphi(x_0)\xi,\xi\right\rangle\leq0$ for any $\xi\in W$ such that $|\xi|=1$. Then
$$\lambda_{N-1}(D^2\varphi(x_0))\leq \max_{\xi\in W:\,|\xi|=1}\left\langle D^2\varphi(x_0)\xi,\xi\right\rangle\leq0 $$
as we wanted to show.
The case $p\geq1$ is included in (\ref{3}).

\noindent\textbf{(\ref{3})} For $p\geq1$ we make use of Lemma \ref{rem1} with $f(u)=u^p$. Indeed for $p>1$  the function
$$
u(|x|)=\left[\frac{2k}{(p-1)(\mu+|x|^2)}\right]^\frac{1}{p-1}
$$
satisfies \eqref{eqremark} for any $\mu>0$. A similar conclusion holds in the case $p=1$ with $$u(|x|)=\mu\exp\left(-\frac{|x|^2}{2k}\right).$$
\textbf{(\ref{4})} By contradiction let $u$ be a positive viscosity supersolution of \eqref{eq1} and $p\in (0,1)$. Let $v=\frac{1}{1-p}u^{1-p}$ as in \cite{BK}. Then 
$$D^2v=-pu^{-p-1}\grad u\otimes \grad u +u^{-p}D^2u\leq u^{-p}D^2u$$
and
$$\Pmk(D^2v)\leq -1\quad \mbox{in}\quad \R^N.$$
This in particular implies that for any $R>0$ the unique solution of
$$
\left\{\begin{array}{lc}
-\Pmk(D^2 w_R)=1 & \mbox{in}\ B_R\\
w_R=0 & \mbox{on}\ \partial B_R
\end{array}
\right.
$$
satisfies $w_R\leq v$ for any $R$. But this is a contradiction since $w_R(x)=-\frac{|x|^2}{2k}+\frac{R^2}{2k}$.
\end{proof}

\begin{remark}
\rm Concerning Theorem \ref{P1}-(\ref{4}), it is worth to point out that the nonexistence of entire supersolutions continues to hold for $p\leq0$. By contradiction let $u$ be a positive supersolution of \eqref{eq1}. Then  $v=\frac{1}{u}$ satisfies in the viscosity sense
$$
D^2v=-\frac{D^2u}{u^2}+\frac{2}{u^3}\grad u\otimes\grad u\geq-\frac{D^2u}{u^2}
$$ 
and, using the monotonicity of $\Ppk$, 
$$
\Ppk(D^2v)-v^{2-p}=0\quad\mbox{in}\quad\Rn.
$$ 
Since $p\leq0$, we obtain a contradiction to the the Keller-Osserman type result of \cite[Theorem 1.1 and Corollary 3.6]{CDLV1}.
\end{remark}

\subsection{Entire solutions of  $\Ppk(D^2u)+u^p=0$.}
\begin{proof}[\rm\textbf{Proof of Theorem \ref{P2}}]
\noindent\textbf{(\ref{6})} 
Let $2<k<N$ and let $0<p\leq\frac{k}{k-2}$. Following \cite[Theorem 4.1]{CL}, we suppose by contradiction that $u>0$ is a viscosity supersolution of 
\begin{equation}\label{eq10}
\Ppk(D^2u(x))+u^p(x)=0\quad\mbox{in}\quad \Rn.
\end{equation}
For $r>0$ let $\varphi(|x|)=m(r)\left[1-\frac{{\left[(|x|-r)^+\right]}^3}{r^3}\right] $. The difference $u-\varphi$ attains its minimum at a point $x_r$ such that $r\leq|x_r|<2r$. Using \eqref{monotonicity},\eqref{eq10} we have
\begin{equation}\label{eq11}
m^p(2r)\leq u^p(x_r)\leq -k\frac{\varphi'(|x_r|)}{|x_r|}\leq 3k\frac{m(r)}{r^2}\leq C\frac{m(2r)}{r^2}
\end{equation}
where $C=3k2^{k-2}$. If $0<p\leq1$ we obtain the contradiction
$$
r^2\leq Cm^{1-p}(2r)\leq Cu^{1-p}(0)\quad \forall r>0.
$$
If $p>1$ then from \eqref{eq11}
\begin{equation}\label{eq12}
m(2r)(2r)^{k-2}\leq\frac{C}{r^{\frac{k-p(k-2)}{p-1}}},
\end{equation}
with $C={(3k2^{p(k-2)})}^{\frac{1}{p-1}}$, and again we get a contradiction if $p<\frac{k}{k-2}$, since the right hand side of \eqref{eq12} tends to 0 while  $m(2r)(2r)^{k-2}$ is a positive nondecreasing function. For $p=\frac{k}{k-2}$ we have the bound
\begin{equation}\label{eq13}
m(2r)(2r)^{k-2}\leq C.
\end{equation}
Let us consider the smooth radial function $\psi(|x|)=\alpha\frac{\log|x|}{|x|^{k-2}}+\beta$, where $\alpha>0$ and $\beta$ are to be  suitably chosen in order to compare $\psi$ and $u$ in $\Omega=\left\{x\in\Rn:r_1<|x|<r_2\right\}$. First note that for $|x|>r_1:=\exp\left(\frac{2k-3}{(k-1)(k-2)}\right)$ the function $\psi$ is convex and decreasing. Hence, it is a solution for $|x|>r_1$ of the equation
\begin{equation}\label{eq14}
\Ppk(D^2\psi(|x|))=\psi''(|x|)+(k-1)\frac{\psi'(|x|)}{|x|}=-\alpha\frac{(k-2)}{|x|^k}.
\end{equation} 
Moreover 
$$
\Ppk(D^2u(x))\leq-m^{\frac{k}{k-2}}(|x|)\leq-\frac{m^{\frac{k}{k-2}}(r_1)r_1^k}{|x|^k}\quad\mbox{for}\quad|x|>r_1.
$$
We pick 
$$
\alpha=\min\left\{\frac{m(r_1)-m(r_2)}{\frac{\log r_1}{r_1^{k-2}}-\frac{\log r_2}{r_2^{k-2}}}, \frac{m^{\frac{k}{k-2}}(r_1)r_1^k}{k-2}\right\}\quad\mbox{and}\quad \beta=m(r_2)-\alpha\frac{\log r_2}{r_2^{k-2}}.
$$
In this way $\Ppk(D^2u)\leq\Ppk(D^2\psi)$ in $\Omega$ and $u\geq\psi$ on $\partial\Omega$. By comparison we have
\begin{equation}\label{eq16}
m(2r)\geq\alpha\frac{\log(2r)}{(2r)^{k-2}}+\beta\quad\mbox{for}\quad r_1<2r<r_2.
\end{equation}
Taking into account that $\lim_{r_2\to+\infty}m(r_2)=0$, as a consequence of the bound \eqref{eq13}, and sending $r_2\to+\infty$ in \eqref{eq16} we obtain 
$$
m(2r)\geq\alpha_\infty\frac{\log(2r)}{(2r)^{k-2}}\quad\mbox{for any}\quad r>\frac{r_1}{2},
$$
with $\alpha_\infty=\lim_{r_2\to+\infty}\alpha>0$. This inequality is in contradiction with  \eqref{eq13}.

\noindent\textbf{(\ref{7})} Let $\mu>0$, $\beta\in[\frac{1}{p-1},\frac{k-2}{2})$ and let $u(|x|)=C(\mu+|x|^2)^{-\beta}$, where $C=C(\beta,k,p)$ is a positive constant to be determined. Since $u''(|x|)\geq \frac{u'(|x|)}{|x|}$, $u$ satisfies
\begin{equation*}
\begin{split}
\Ppk(D^2u(|x|))+u^p(|x|)&=u''(|x|)+\frac{k-1}{|x|}u'(|x|)+u^p(|x|)\\
&=C(\mu+|x|^2)^{-(\beta+1)}\left[\frac{4\beta(\beta+1)|x|^2}{\mu+|x|^2}-2\beta k+\frac{C^{p-1}}{(\mu+|x|^2)^{\beta(p-1)-1}}\right]\\
&\leq C(\mu+|x|^2)^{-(\beta+1)}\left[2\beta(2\beta+2-k)+\frac{C^{p-1}}{(\mu+|x|^2)^{\beta(p-1)-1}}\right]\\&\leq0
\end{split}
\end{equation*}
if $C=\mu^{\beta-\frac{1}{p-1}}\left(2\beta(k-2-2\beta)\right)^{\frac{1}{p-1}}$.

\noindent\textbf{(\ref{8})} %If $p=\frac{k+2}{k-2}$ just take $u(|x|)=C(\mu+|x|^2)^{-\frac{k-2}{2}}$ and $C=(k(k-2)\mu)^\frac{k-2}{4}$. 
For $p\geq\frac{k+2}{k-2}$ let $u=u(|x|)$ be a radial solution of $\Delta u+u^p=0$ in $\mathbb R^k$, namely 
\begin{equation}\label{eq15-}
\begin{cases}
u''(r)+\frac{k-1}{r}u'(r)=-u^p(r) & \text{for $r>0$}\\
u'(0)=0. &
\end{cases}
\end{equation}
We first note that $u''\geq\frac{u'}{r}$ for any $r>0$. In fact setting $v(r)=u''(r)-\frac{u'(r)}{r}$ and using \eqref{eq15-} 
\begin{equation*}
\begin{split}
v'(r)&=\left(-u^p(r)-k\frac{u'(r)}{r}\right)'\\
&=-pu^{p-1}(r)u'(r)-\frac kr v(r)\\
&\geq-\frac kr v(r)
\end{split}
\end{equation*}
where in the last inequality we used that $u'\leq0$, being $u(|x|)$  in particular a radial superharmonic function. Hence the function $v(r)r^k$ is nondecreasing and $v(r)\geq0$ for any $r\geq0$. \\
Since $u''(r)\geq\frac{u'}{r}$, then for any $x\in\mathbb R^N$
$$
\Ppk(D^2u(|x|))+u^p(|x|)=u''(|x|)+\frac{k-1}{|x|}u'(|x|)+u^p(|x|)=0.
$$
In the case $p=\frac{k+2}{k-2}$ let us explicitly remark that 
$$
u(|x|)=\frac{C}{(\mu+|x|^2)^{\frac{k-2}{2}}},
$$
where $\mu>0$ and $C=(k(k-2)\mu)^{\frac{k-2}{4}}$.

\noindent\textbf{(\ref{9})} Assume by contradiction that such a solution exists. If for some $r>0$ the inequality $u''(r)<\frac{u'}{r}$ holds, then by the equation we have $k\frac{u'(r)}{r}=-u^p(r)$. Reasoning as in the Lemma \ref{rem1} we immediately obtain the contradiction $u''\geq\frac{u'}{r}$. Hence  $u''(r)\geq\frac{u'}{r}$ and 
\begin{equation}\label{eq15} 
u''(r)+\frac{(k-1)u^\prime}{r}=-u^p.
\end{equation}
Now observe that a  positive solution of \eqref{eq15} would be an entire radial solution in $\R^k$ of
$$\Delta u+u^p= 0$$
which do not exists for $p<\frac{k+2}{k-2}$ by the classical result of Gidas and Spruck \cite[Theorem 1.1]{gs1}.
\end{proof}

\section{Solutions in the half space}\label{Half space}

\begin{proof}[\rm\textbf{Proof of Theorem \ref{10}}]
For any positive $\alpha$ consider the initial value problem
\begin{equation*}
\begin{cases}
v''(t)+v^p(t)=0, & t>0\\
v(0)=0\\
v'(0)=\alpha\,.
\end{cases}
\end{equation*}
Such a problem admits classical solutions depending on $p$ and $\alpha$, which are increasing in $[0,\tau]$, where $\tau=\tau(\alpha,p)$ and $v'(\tau)=0$, then decreasing in $[\tau,T]$ with $v(T)=0$ and $T=T(\alpha,p)$. Let us denote by $v_T(t)=v_T(t;\alpha,p)$ the $T$-periodic extension of $v$.\\
Now fix $\alpha>0$ and define $u(x)=v_T(x_N)$ for any $x\in\mathbb R^N_+$. By construction $u=0$ on $\partial\mathbb R^N_+$ and 
$$
\lim_{n\to+\infty}u(x',nT+\tau)=v(\tau)=\max_{[0,T]}v>0
$$
uniformly with respect to $x'\in\mathbb R^{N-1}$. Hence $u$ satisfies \eqref{limsup}. Moreover in the set $\mathbb R^N\backslash\Gamma$, where $\Gamma=\left\{(x',nT)\in\mathbb R^N_+:\,n\in\mathbb N\right\}$,  $u$ is a twice differentiable function such that
$$
D^2u(x)={\rm diag}[0,\ldots,0,v_T''(x_N)],\quad x\in\mathbb R^N\backslash\Gamma.
$$
Since  $v_T''(x_N)<0$ then 
$$
\Pmk(D^2u(x))+u^p(x)=v_T''(x_N)+v_T^p(x_N)=0,\quad x\in\mathbb R^N\backslash\Gamma.
$$
To complete the proof we show that $u$ satisfies  equation \eqref{halfspace} also in $\Gamma$ in the viscosity sense. By periodicity it is sufficient to treat only the case $x_N=T$. First we notice that $u$ is trivially a viscosity subsolution since there are no $C^2$ test function touching $u$ from above at $(x',T)$. Hence, we have only to prove the supersolution property. Let $\varphi$ be a test function touching $u$ from below  at $(x',T)$. Since the restriction of $\varphi$ to the $(N-1)$-dimensional affine subspace orthogonal to the $x_N$-axis through $(x',T)$ has a maximum point in $(x',T)$ equal to zero, i.e.
$$
\varphi(x'+t\xi,T)\leq\varphi(x',T)=u(x',T)=0
$$
for $t$ small enough and $\xi\in\mathbb R^{N-1}$,
 we can use the Courant-Fischer formula, as in the proof of Theorem \ref{P1}-(\ref{2}), to conclude that the first $N-1$ eigenvalues of $D^2\varphi(x',T)$ are nonpositive. Since $k<N$, this implies
$$
\Pmk(D^2\varphi(x',T))\leq0=-\varphi^p(x',T).
$$
It is worth to point out that by scaling the function $u$ by $\tilde u=\gamma^{\frac{2}{p-1}}u(\gamma x)$ we obtain solutions of \eqref{halfspace} satisfying for any $c>0$ the condition $\displaystyle\limsup_{x_N\to+\infty}u(x',x_N)=c$.   

\end{proof}

\begin{lemma}\label{lemmasub}
For $k<N$, the function $\varphi(x)=\frac{x_N}{|x|^k}$ is a classical subsolution of
$$
\Ppk(D^2\varphi(x))=0\quad{\text{in $\mathbb R^N_+$}}
$$
vanishing on $\partial\mathbb R^N_+\backslash\left\{0\right\}$.
\end{lemma}
\begin{proof}
Using Lemma \ref{eigenvalues} 
\begin{equation*}
\begin{split}
\lambda_1(D^2\varphi(x))&=-k\frac{x_N}{|x|^{k+2}}+\frac{k^2\frac{x_N}{|x|^{k+2}}-\sqrt{\left(k^2\frac{x_N}{|x|^{k+2}}\right)^2+4\frac{k^2}{|x|^{2k+2}}\left(1-\left(\frac{x_N}{|x|}\right)^2\right)}}{2}\\
\lambda_2(D^2\varphi(x))&=\ldots=\lambda_{N-1}(D^2\varphi)=-k\frac{x_N}{|x|^{k+2}}\\
\lambda_N(D^2\varphi(x))&=-k\frac{x_N}{|x|^{k+2}}+\frac{k^2\frac{x_N}{|x|^{k+2}}+\sqrt{\left(k^2\frac{x_N}{|x|^{k+2}}\right)^2+4\frac{k^2}{|x|^{2k+2}}\left(1-\left(\frac{x_N}{|x|}\right)^2\right)}}{2}.
\end{split}
\end{equation*} 
Since $k<N$, then
\begin{equation*}
\begin{split}
\Ppk(D^2\varphi(x))%&=-k^2x_N|x|^{-k-2}+\frac{k^2x_N|x|^{-k-2}+\sqrt{k^2(k^2+4)x_N^2|x|^{-2k-4}+4k^2x_N|x|^{-2k-3}}}{2}\\
&=-k^2\frac{x_N}{|x|^{k+2}}+\frac{k^2\frac{x_N}{|x|^{k+2}}+\sqrt{\left(k^2\frac{x_N}{|x|^{k+2}}\right)^2+4\frac{k^2}{|x|^{2k+2}}\left(1-\left(\frac{x_N}{|x|}\right)^2\right)}}{2}\\
%&\geq\frac12\left(-k^2x_N|x|^{-k-2}+k\sqrt{k^2+4}x_N|x|^{-k-2}\right)\\
&\geq0.
\end{split}
\end{equation*}
\end{proof}

Let us consider now nonnegative viscosity supersolution $u\in LSC(\mathbb R^N_+)$ of
\begin{equation}\label{4eq1}
\Ppk(D^2\varphi(x))=0\quad{\text{in $\mathbb R^N_+$}}.
\end{equation}
By strong minimum principle Theorem \ref{SMP}, either $u\equiv0$ or $u>0$. Let us fix our attention on the case $u>0$. Eventually replacing $u(x)$ by $u(x',x_N+\varepsilon)$ for positive $\varepsilon$, we assume from now on  that $u>0$ in $\overline{\mathbb R^N_+}$. Let
\begin{equation}\label{mu}
\mu(r)=\inf_{\overline B_r\cap\mathbb R^N_+}\frac{u}{x_N}
\end{equation} 
for $r>0$. Since $u>0$ on $\partial\mathbb R^N_+$, the infimum is in fact a minimum and there exists $\hat{x}=\hat{x}(r)$ such that $\mu(r)=\frac{u(\hat{x})}{\hat{x}_N}$. We infer that $\hat{x}\in\partial B_r\cap\mathbb R^N_+$. Indeed let us assume by contradiction that $\hat{x}\in B_r\cap\mathbb R^N_+$. Then  the function $v(x)=u(x)-\mu(r)x_N$ is nonnegative in  $B_r\cap\mathbb R^N_+$, $v(\hat{x})=0$  and  
$$
\Ppk(D^2v)=\Ppk(D^2u)\leq0\quad\text{in}\quad B_r\cap\mathbb R^N_+.
$$
The strong minimum principle yields the contradiction $u\equiv0$. In particular  $\frac{u(x)}{x_N}$ cannot achieves an interior minimum point in $B_r\cap\mathbb R^N_+$, hence $\mu(r)$ is a decreasing function for $r>0$.

Let $0<r_1<r_2$ and let $\varphi(x)=x_N\left(\frac{c_1}{|x|^k}+c_2\right)$  for $x\in (B_{r_2}\backslash B_{r_1})\cap \mathbb R^N_+$. By Lemma \ref{lemmasub},   we have $\Ppk(D^2\varphi)\geq0$ for $c_1\geq0$. Choosing $c_1,c_2$ in such that
$
u\geq\varphi
$
on $\partial((B_{r_2}\backslash B_{r_1})\cap \mathbb R^N_+)$,  we obtain by comparison the following version of Hadamard three circles theorem in the half space, see \cite[Theorem 2.7]{L} for details.

\begin{theorem}\label{Hadamard half space}
Let $u\in LSC(\overline{\mathbb R^N_+})$ be a nonnegative viscosity supersolution of $\Ppk(D^2u)=0$ in $\mathbb R^N_+$. Then for every fixed $0<r_1<r_2$ and any $r_1\leq r\leq r_2$ one has
\begin{equation}
\mu(r)\geq\frac{\mu(r_1)(r^{-k}-r_2^{-k})+\mu(r_2)(r_1^{-k}-r^{-k})}{r_1^{-k}-r_2^{-k}}.
\end{equation} 
In particular the map 
\begin{equation}\label{monotonicity 2}
r\in(0,+\infty)\mapsto\mu(r)r^k\quad\text{is nondecreasing.}
\end{equation}
\end{theorem}

\begin{proof}[\rm\textbf{Proof of Theorem \ref{Pk+ halfspaces}}]
Observe that 
$$
\Ppk(D^2u)+u^p\leq0\quad\text{in $\R_+^N$}\quad\Longrightarrow\quad \Delta u+u^p\leq0\quad\text{in $\R_+^N$}.
$$
Hence for $p\leq\frac{N+1}{N-1}$ the nonexistence result is proved (see e.g. \cite{KLS}).\\
Fix $p>\frac{N+1}{N-1}$ and assume by contradiction that $u>0$
 in $\overline{\mathbb R^N_+}$ is a viscosity supersolution of 
\begin{equation}\label{eq1half}
\Ppk(D^2u)+u^p=0\quad{\text{in}}\quad\mathbb R^N_+.
\end{equation}
For $r>0$ let 
$$
\varphi(|x-(0',3r)|)=\left(\min_{\overline B_r((0',3r))}u\right)\left[1-\frac{\left[(|x-(0',3r)|-r)^+\right]^3}{r^3}\right].
$$
Following \cite[Theorem 3.1]{L}
$$
\min_{\mathbb R^N_+}(u-\varphi)=(u-\varphi)(\tilde x)\leq0,
$$
where $\tilde x=\tilde x(r)$ and $r\leq|\tilde x-(0',3r)|<2r$. Using $\varphi$ as test function in \eqref{eq1half} we get
\begin{equation*}
\begin{split}
u^p(\tilde x)\leq-\Ppk(D^2\varphi(\tilde x))&=-k\frac{\varphi'(|\tilde x-(0',3r)|)}{|\tilde x-(0',3r)|}\\&=\frac{3k}{r^3}\frac{(|\tilde x-(0',3r)|-r)^2}{|\tilde x-(0',3r)|}\left(\min_{\overline B_r((0',3r))}u\right)\leq \frac{3k2^{k-2}}{r^2}u(\tilde x),
\end{split}
\end{equation*}
where in the last inequality we used the monotonicity property expressed by Proposition \ref{eq7}. In this way
\begin{equation}\label{eq2half}
u^{p-1}(\tilde x)\leq\frac{3k2^{k-2}}{r^2}\quad\forall r>0.
\end{equation}
Since 
$$
u(\tilde x)=\frac{u(\tilde x)}{\tilde x_N}\tilde x_N\geq r\mu(5r),
$$
from \eqref{eq2half} we obtain
\begin{equation}\label{eq3half}
r^k\mu(r)\leq\frac{C}{r^{\frac{p+1}{p-1}-k}},
\end{equation}
where $C=(3k2^{k-2}5^{p+1})^{\frac{1}{p-1}}$, which is in contradiction with the monotonicity property \eqref{monotonicity 2} if $k=1$ or  if $k>1$ and $p<\frac{k+1}{k-1}$. %In the limit case $p=\frac{k+1}{k-1}$, \eqref{eq3half} gives the upper bound
%\begin{equation}\label{upperboundhalf}
%r^k\mu(r)\leq C\quad\forall r>0.
%\end{equation}
%Let $\varphi(x)=x_N\left[\alpha\frac{\log|x|}{|x|^k}+\beta\right]$, for $\alpha\geq0$ and $\beta$ to be determined. With the notations of Lemma \ref{eigenvalues} and using the fact that $k<N$, one has
%\begin{equation*}
%\begin{split}
%\Ppk(D^2\varphi(x))&=kd+\frac{b+2c\frac{x_N}{|x|}+\sqrt{\left(b+2c\frac{x_N}{|x|}\right)^2+4c^2\left(1-\left(\frac{x_N}{|x|}\right)^2\right)}}{2}\\
%&\geq kd+\frac{b+2c\frac{x_N}{|x|}}{2}\\
%&=\frac{2\alpha x_N}{|x|^{k+2}}\left(k(k-2)\log |x|-2(k-1)\right)\\
%&\geq0\quad\quad\text{for $r_1:=\exp\left(\frac{2(k-1)}{k(k-2)}\right)$}.
%\end{split}
%\end{equation*}
%For any $r_2>r_1$, we pick 
%$
%\alpha=\frac{\mu(r_1)-\mu(r_2)}{\frac{\log r_1}{r_1^k}-\frac{\log r_2}{r_2^k}} 
%$ and $\beta=\mu(r_2)-\alpha\frac{\log r_2}{r_2^k}$, so that $$u\geq\varphi\quad\text{on $\partial\left(\left(B_{r_2}\backslash B_{r_1}\right)\cap \mathbb R^N_+\right)$}.$$
%Since $\Ppk(D^2u)\leq0$ in particular, by comparison we deduce 
%$$
%\mu(r)\geq\alpha\frac{\log r}{r^k}+\beta\quad\forall r_1<r<r_2.
%$$
%Letting $r_2\to+\infty$ we obtain
%$$
%\mu(r)r^k\geq\frac{\mu(r_1)}{\frac{\log r_1}{r_1^k}}\log r
%$$
%and this contradicts the upper bound  \eqref{upperboundhalf}.
\end{proof}

\end{document}